\documentclass[11pt]{amsart}
\usepackage{amssymb,amsmath,amscd}
\usepackage{latexsym,bm,bbm,mathrsfs}
\usepackage{hyperref,graphicx}
\usepackage{mathtools}
\usepackage{stmaryrd}
\usepackage[all,pdf]{xy}
\usepackage{tikz}
\usepackage{tikz-cd}
\usetikzlibrary{positioning}

\newtheorem{thm}{Theorem}

\newtheorem{prop}[thm]{Proposition}

\newtheorem{lem}[thm]{Lemma}

\newcommand{\C}{{\mathbb{C}}}

\newcommand{\N}{\mathbb{N}}

\newcommand{\R}{{\mathbb{R}}}
\newcommand{\Z}{{\mathbb{Z}}}

\newcommand{\GL}{\operatorname{GL}}
\newcommand{\SL}{\operatorname{SL}}

\newcommand{\sgn}{\operatorname{sgn}}

\newcommand{\X}{\mathsf{X}}

\def\det{\text{det}}

\def\R{\mathbb{R}}

\def\Z{\mathbb{Z}}

\def\C{\mathbb{C}}
\def\N{\mathbb{N}}

\def\GL{\text{GL}}
\def\SL{\text{SL}}

\def\Pr{\text{Pr }}

\def\1{\mathbbm{1}}
\def\X{\mathsf{X}}

\def\<{\langle}
\def\>{\rangle}

\title[Tensor power asymptotics]{Tensor Power Asymptotics for Linearly Reductive Groups}

\author{Michael Larsen}

\address{ Department of Mathematics, Indiana University, Bloomington, IN 47405,
 U.S.A.}
\email{mjlarsen@iu.edu}
\thanks{The author was partially supported by NSF grant DMS-2401098 and the Simons Foundation.}

\begin{document}

\begin{abstract}
Given a finite-dimensional faithful representation $V$ of a
linearly reductive group $G$ over a field $K=\bar K$, 
we consider the growth of the number of irreducible factors of $V^{\otimes n}$ when $n$ is large.
We prove that there exist upper and lower bounds which are constant multiples of $n^{-u/2} (\dim V)^n$,
where $u$ is the dimension of any maximal unipotent subgroup of $G$.

\end{abstract}

\maketitle

Let $G$ be a linear algebraic group over an algebraically closed field $K$ of characteristic $p\ge 0$.
and $\rho\colon G\to \GL(V)$ denote a finite dimensional representation of $G$. Following \cite{COT}, we define $b_n^{G,V}$
to be the number of $G$-indecomposable factors in $V^{\otimes n}$. 

In \cite{CEOT}, Coulembier, Etingof, Ostrik, and Tubbenhauer asked whether there exist constants $r\ge 0$ and $B>A>0$ such that for large $n$, 
\begin{equation}
\label{hope}
A n^r (\dim V)^n < b_n^{G,V} <B n^r (\dim V)^n.
\end{equation}
They proved that this is so when $G=\SL_2$ and $V$ is $2$-dimensional.
For $p=2$, Larsen \cite{Larsen} proved the same result and proved also the desired lower bound for any tilting module $V$ of $\SL_2$.
For $p$ odd, Sheu \cite{Sheu} proved both the upper and lower bounds for all tilting modules $V$. These results are particularly striking because the powers
$r$ appearing in these papers are transcendental and seem to arise from some kind of fractal dimension, whereas in the $p=0$ case for $\SL_2$, it is easy to see $r=-1/2$. There seems to be an expectation that $r$ should have
a dimensional interpretation, and this paper explains why this is so in the most favorable situation, where the indecomposables are all irreducible. 

I was under the impression that
Theorem~\ref{main}, was known or could at least be deduced immediately from known results, but although a good deal \emph{is} known about asymptotics of tensor powers (see, e.g., \cite{Biane,COT,PR,TZ} and the references contained therein), this seems to be a genuine gap in the literature.

Following Mumford \cite{GIT}, we say $G$ is \emph{linearly reductive} if every representation $V$ of $G$ is
completely reducible. If $p=0$, this condition is equivalent to being reductive. If $p>0$, by a theorem of Nagata \cite[Th\'eor\`eme~IV.3.3.6]{DG},
$G$ is linearly reductive if and only if $G^\circ$ is a torus and $G/G^\circ$ is of prime-to-$p$ order.

We define $u$ to be the dimension of any maximal unipotent subgroup of $G$ or, equivalently, the number of positive roots of $G$. The main result of this paper is the following:

\begin{thm}
\label{main}
If $V$ is faithful,
there exist positive constants $A$ and $B$ such that for all $n$ sufficiently large,
\begin{equation}
\label{basic bound}
A n^{-u/2} (\dim V)^n < b_n^{G,V} < B n^{-u/2}(\dim V)^n.
\end{equation}
\end{thm}

In fact, this is true in slightly greater generality, namely whenever $\ker \rho$ is a central subgroup of $G$. This hypothesis is certainly needed, since the numbers $b_n^{G,V}$ cannot distinguish $G$ from $G/\ker \rho$, and if $\ker\rho$ is not central, it must have non-trivial unipotent radical, implying that $u$ is strictly smaller for $G/\ker\rho$ than for $G$.

One might ask for more than \eqref{hope}, namely an asymptotic formula of the form $b_n^{G,V} \sim C n^{-u/2} (\dim V)^n$. We know from \cite{CEOT,Larsen,Sheu} that this is too much to
hope for in positive characteristic even in the case $G = \SL_2$. Even in characteristic zero, it is not generally true, for instance when $G$ is the dihedral group of order $8$ and $V$ is its
irreducible $2$-dimensional representation.

The first step in proving Theorem~\ref{main} 
is to reduce to the case that $K=\C$ and $G$ is connected. The possible irreducible factors are then indexed by positive weights. We use a local limit theorem to estimate weight multiplicities for $V^{\otimes n}$
and then use the Weyl character formula 
and the method of finite differences to deduce irreducible character multiplicities. Summing, we obtain an estimate for $b_n^{G,V}$. 
From this point of view, one can give a simple heuristic explanation of Theorem~\ref{main}. Namely, by the central limit theorem, a typical weight of $V^{\otimes n}$ has length of order $\sqrt n$,
and by the Weyl dimension formula, the degree of the irreducible representation associated to a typical weight of order $\sqrt n$ is of order $n^{u/2}$.

Biane \cite{Biane} gave an asymptotic formula for the individual multiplicities of irreducible representations in $V^{\otimes n}$ in the central limit theorem region (i.e., when the length of the highest weight is $O(\sqrt n)$.) This is not quite enough for our purposes since we need the general connected reductive case and also need bounds for multiplicity when the highest weight is somewhat larger than $\sqrt n$. Rather than trying to deal with 
these difficulties using his method we have chosen the slightly different approach described above, which was suggested, but not implemented, by Tate and Zelditch \cite{TZ}.
Recent work of Postnova and Reshetikhin \cite{PR} has improved on Biane in certain respects, in particular giving good estimates of irreducible character multiplicities outside the central limit theorem region. Its goal, however, is somewhat different than the one pursued here, focusing on the Plancherel measure, in which irreducible characters are weighted proportionally to their degrees.
The problem of counting irreducible characters in this weighted sense would be trivial. 
\begin{lem}
Let $G$ be a linearly reductive group and $G^\circ$ its identity component. Then \eqref{basic bound} holds for $G$ if and only if it holds for $G^\circ$.
\end{lem}

\begin{proof}
Indeed, by \cite{Clifford}, the restriction of a $G$-irreducible summand of $V^{\otimes n}$ to $G^\circ$ decomposes into no more than
$G/G^\circ$ irreducible factors.
\end{proof}

We assume henceforth that $G$ is connected. In the positive characteristic case, that means that $G$ is a torus, so every
$G$-representation decomposes into $1$-dimensional irreducibles, so $b_n^{G,V} = (\dim V)^n$, and the theorem is obvious. 
We therefore assume for the rest of the paper that $p=0$.

As $G$ is reductive and $\ker \rho$ is central, it follows that the unipotent part of $\ker\rho$ is trivial, so the homomorphism $G\to G/\ker\rho$
maps unipotent subgroups of $G$ isomorphically to unipotent subgroups of $G/\ker\rho$.
Conversely, starting with a unipotent subgroup $U$ of $G/\ker\rho$, the identity component of its inverse in $G$ is nilpotent and therefore has a subgroup which
 maps isomorphically to $U$ \cite[Theorem~10.6 (3)]{Borel}. This bijective correspondence between unipotent subgroups of $G$ and $G/\ker\rho$ means that $u$ is the same for $G$ and $G/\ker\rho$, and of course
 $b_n^{G,V} = b_n^{G/\ker\rho,V}$ are also the same for all $n$,
 so we may freely replace $G$ by $G/\ker\rho$.

\begin{lem}
It suffices to prove Theorem~\ref{main} under the hypothesis that $\det\,V = 1$.
\end{lem}

\begin{proof}
Let $Z$ denote the center of $G$, $Z^\circ$ its identity component, and $D$ the derived group of $G$.  
For each integer $m>0$, let $\phi_m\colon G\times Z^\circ\to G/D$ denote the homomorphism $\phi_m(g,z) = \overline{g+mz}$,
and let $G_m = \ker \phi_m$. 

Let $X_Z$ denote the character group of $Z^\circ$, and let $x_1,x_2,\ldots,x_{\dim V}\in X_Z$ denote the characters of the restriction of $V$ to $Z^\circ$,
with repetitions allowed.  Let $s =  x_1+\cdots+x_{\dim V}\in X_Z$, and let $W_s$ denote the $1$-dimensional representation of $Z^\circ$ with with weight $s$.
Thus, the restriction of $\det\, V$ to $Z^\circ$ has weight $w$. Let $V_m = V\boxtimes 1|_{G_m}$ denote the pullback of $V$ to $G_m$.
Thus, $\det\,V_m$ is the $m$th tensor power of the restriction of $1\boxtimes W_s$ to $G_m$.
It follows that $V_m\otimes (1\boxtimes W_s^*)|_{G_m}$ has determinant $1$.  Replacing $(G,V)$ by 
$$(G_m,V_m\otimes (1\boxtimes W_s^*)|_{G_m}),$$
we may assume our representation has determinant $1$.
\end{proof}

Applying this, we assume from now on that $\det\,V = 1$. Replacing $G$ by $G/\ker\rho$, we may further assume that $V$ is faithful. 

Extending scalars from $K$ to some larger algebraically closed field does not change the values of $b_n^{G,V}$,
since these are determined by which flag varieties $\GL(V^{\otimes n})/P$ have fixed points under the action of $G$,
and the existence of such a fixed point does not depend on which algebraically closed field we work over.
Therefore, we may assume $K=\C$.  Let $G^c$ denote a maximal compact subgroup  of $G$, which is a connected Lie group since $G$ is connected.

Let now $T$ denote a maximal torus of $G^c$ and $X$ the group of characters of $T$. Let $\Phi\subset X$ denote the root system of $G^c$ with respect to $T$. 
Fix a Weyl chamber, and let  $\Phi^+$ denote the set of positive roots of $\Phi$ with respect to this chamber.
Let $W$ denote the Weyl group of $G^c$ with respect to $T$, $\delta$ half the sum of the elements of $\Phi^+$, and $\Lambda\subset X$ the set of
dominant weights.
Let $\chi_V\in \Z[X]$ denote the formal character of $V$. 

By the Weyl character formula for connected compact Lie groups \cite[VI~(1.7)]{Bt}, if $\chi_\lambda$ denotes the character of the irreducible representation of $G$ with highest weight $\lambda\in\Lambda$, then
$$\chi_\lambda\prod_{\alpha\in \Phi^+} (1-[\alpha]) = \sum_{w\in W} \sgn(w) [w(\lambda+\delta)].$$
Therefore, writing
$$\chi_V = \sum_{\lambda\in\Lambda} a_\lambda^V \chi_\lambda,$$
we obtain
\begin{equation}
\label{formal}
(\chi_V^{\otimes n}) \prod_{\alpha\in \Phi^+} (1-[\alpha]) = (\chi_V)^n \prod_{\alpha\in \Phi^+} (1-[\alpha]) 
= \sum_{\lambda\in\Lambda} a_\lambda^{V^{\otimes n}}  \sum_{w\in W} \sgn(w) [w(\lambda+\delta)].
\end{equation}
Since $b_n^{G,V}$ is the sum over dominant weights $\lambda$ of $a_\lambda^{V^{\otimes n}}$, we can write
$$b_n^{G,V} = \sum_{\lambda\in \delta+\Lambda} c_{\lambda}^{G,V^{\otimes n}},$$
where $c_{\lambda}^{G,V^{\otimes n}} = a_{\lambda-\delta}^{V^{\otimes n}}$ denotes the
$[\lambda]$-coefficient of  \eqref{formal}.

If $\chi_V = \sum_{\chi\in S} p_\chi [\chi]$, we let $\X_1,\X_2,\X_3,\ldots$ be a sequence of i.i.d. random variables, with values in $S\subset X$,
such that $\Pr[\X_i=\chi] = p_\chi/\dim V$.
This distribution is chosen so that
the multiplicity of $\chi$ as a weight of $V^{\otimes n}$ is 
\begin{equation}
\label{prob mult}
(\dim V)^n \Pr[\X_1+\cdots+\X_n=\chi].
\end{equation}
Taking a finite  Edgeworth expansion \cite[Theorem 22.1]{BR}, for every positive integer $k$, we get an estimate of the form
\begin{equation}
\label{Edgeworth}
\Pr[\X_1+\cdots+\X_n=\chi] = n^{-r/2}\biggl(e^{-Q(\chi)/2n}\biggl(\sum_{i=0}^k n^{-i/2}P_i(\chi/\sqrt n\biggr) + o(n^{-k/2})\biggr),
\end{equation}
where $Q$ is a positive-definite quadratic form on $X\otimes \R \cong \R^r$, $P_0$ is a positive constant, and $P_1,\ldots,P_k$ are polynomial functions on $X\otimes\R$.
Defining $P(x_0,x_1,\ldots,x_r) = \sum_{i=0}^k x_0^i P_i(x_1,\ldots,x_r)$, we can rewrite this as
$$\Pr[\X_1+\cdots+\X_n = \chi] = n^{-r/2} e^{-Q(\chi)/2n} P(n^{-1/2},n^{-1/2}\chi) + o(n^{-r/2-k/2}).$$

\begin{prop}
\label{difference}
Let $Q$ be a positive definite form on $\R^r$, $P(x_0,x_1,\ldots,x_r)$ a non-zero polynomial,  $\vec a = (a_1,\ldots,a_r)\in \Z^r$ a non-zero vector, and $N$ a positive real number.
Suppose $f\colon \N\times \Z^r\to \R$ is a function such that for all $\vec v\in \Z^r$ and all positive integers $n$,
$$f(n,\vec v) = e^{-Q(v)/2n}P(n^{-1/2},n^{-1/2}\vec v)+ o(n^{-N}),$$
Then there exists a polynomial $P'(x_0,x_1,\ldots,x_r)$ such that
$$f(n,\vec v+\vec a) - f(n,\vec v) = n^{-1/2} e^{-Q(v)/2n}P'(n^{-1/2},n^{-1/2}\vec v)+ o(n^{-N}).$$
Moreover,
$$\deg P'(0,x_1,\ldots,x_r) = 1+\deg P(0,x_1,\ldots,x_r).$$
\end{prop}

\begin{proof}
For every $\epsilon > 0$, it suffices to prove this under the assumption that $|\vec v| = O(n^{1/2+\epsilon})$,
since $e^{-n^{2\epsilon}} = o(n^{-N})$.

Up to an error $o(n^{-N})$, we write $f(n,\vec v+\vec a) - f(n,\vec v)$ as
\begin{equation}
\label{two pieces}
\begin{split} & \Bigl(e^{-Q(\vec v+\vec a)/2n} - e^{-Q(\vec v)/2n}\Bigr)P(n^{-1/2},n^{-1/2}(\vec v+\vec a))\\
&\qquad\quad+ e^{-Q(\vec v)/2n} \biggl(P(n^{-1/2},n^{-1/2}(\vec v+\vec a))- P(n^{-1/2},n^{-1/2}\vec v))\biggr).
\end{split}
\end{equation}
The first multiplicand on the left hand side can be written
$$e^{-Q(\vec v+\vec a)/2n} - e^{-Q(\vec v)/2n} = e^{-Q(\vec v)/2n} (e^{-\vec a\cdot \vec v /n- Q(\vec a)/2n}-1),$$
where $\cdot$ is the inner product associated to $Q$. 

Choose $k>2N$. Since $|\vec a\cdot \vec v|/n = O(n^{-1/2+\epsilon})$, 
taking $\epsilon$ small enough, we have 
$$|\vec a\cdot \vec v/n|^{k+1} = o(n^{-k/2}) = o(n^{-N}).$$
Thus,
$$e^{-\vec a\cdot \vec v /n- Q(\vec a)/n}-1 = \sum_{i=1}^k \frac{(-n^{-1/2}\vec a\cdot (n^{-1/2}\vec v) - n^{-1}Q(\vec a))^i}{i!}
+o(n^{-k/2}),$$
which is $R(n^{-1/2},n^{-1/2}\vec v) + o(n^{-N})$, where
$$R(x_0, \vec x) = \sum_{i=1}^k \frac{\bigl(-x_0 (\vec a\cdot \vec x) - Q(\vec a) x_0^2/2\bigr)^i}{i!}.$$
Defining
$$S(x_0,\vec x) =  \frac{R(x_0,\vec x)P(x_0,\vec x+x_0\vec a)}{x_0},$$
we see that $n^{-1/2}S(n^{-1/2},n^{-1/2} \vec v)e^{-Q(\vec v)/2n}$ gives the first term in \eqref{two pieces} 
up to an error term $o(n^{-N})$.  Moreover, $S(0,\vec x) = -(\vec a\cdot \vec x) P(0,\vec x)$, whose degree is $1+\deg P(0,\vec x)$.

Finally, defining
$$S'(x_0,\vec x) = \frac{P(x_0,\vec x+ x_0\vec a) - P(x_0,\vec x)}{x_0},$$
we see that $S'$ is a polynomial, and 
$$n^{-1/2} S'(n^{-1/2},n^{-1/2} \vec v) e^{-Q(\vec v)/2n}$$
gives the second term in \eqref{two pieces}.
As $S'(0,\vec x)$ is a directional derivative of $P(0,\vec x)$ its degree is at most $-1+\deg S(0,\vec x)$.
Adding the two terms in \eqref{two pieces} together, we get $P'$ of the desired form.
\end{proof}

\begin{prop}
\label{c_lambda estimate}
For every $N$, there exists a polynomial $P$ on $\R\times X\otimes \R$ such that 
\begin{equation}
\label{asymptotic}
c_\lambda^{G,V^{\otimes n}} (\dim V)^{-n} = n^{-r/2-u/2} e^{-Q(\lambda)/2n} P(n^{-1/2},n^{-1/2}\lambda) + o(n^{-N}).
\end{equation}
Moreover, $P(0,\vec x)$ has degree $u$ and is non-negative on the dominant cone of $X\otimes \R$
\end{prop}

\begin{proof}
By \eqref{prob mult}, $\Pr[\X_1+\cdots+\X_n = \lambda]$ gives the left hand side of \eqref{asymptotic}.
We choose $N>r+u$ and apply \eqref{Edgeworth} to obtain a polynomial $P$ with error term $o(n^{-N})$.
We then apply Proposition~\ref{difference} $u$ times. Note that multiplying an element of $\Z[X]$ by $(1-[\alpha])$ applies a difference
operator with respect to the vector $\alpha$ to the coefficients.
Each time we do so, we introduce a factor of $n^{-1/2}$, and $\deg P(0,\vec x)$ increases by $1$, leading to the right hand side of \eqref{asymptotic}. 

Every $\vec v$ in the dominant cone of $X\otimes \R$ is the limit of a convergent sequence of vectors of the form $n_i^{-1/2}\lambda_i$, where $\lambda_i$ is a dominant weight, $n_i$ is a positive integer, and $n_i\to \infty$.
We have
$$P(0,\vec v) = \lim_{i\to \infty} P(n_i^{-1/2}, n_i^{-1/2}\lambda_i).$$
If $P(0,\vec v) < 0$, the terms on the right hand side of \eqref{asymptotic} become bounded above by a negative multiple of $n_i^{-r/2-u/2}$ plus $o(n_i^{-N})$, so all but finitely many terms are negative. However, the left hand side is always non-negative.

\end{proof}

\begin{prop}
\label{AB bounds}
Let $P$ be a polynomial in $x_0,\ldots,x_r$ which is non-negative on $[0,\infty)^{r+1}$ and not divisible by $x_0$.
Let $M$ be a semigroup contained in $\N^r$ which spans $\R^n$, and $Q$ be a positive-definite quadratic form on $\R^r$.
There exist constants $B>A>0$ such that for all $t>0$ sufficiently small,
$$A t^{-r} < \sum_{\lambda\in M}  e^{-Q(t \lambda)/2} P(t,t\lambda) < B t^{-r}.$$
\end{prop}

\begin{proof}
For the upper bound, we prove slightly more, namely,
\begin{equation}
\label{absolute ub}
 \sum_{\lambda\in M}  e^{-Q(t \lambda)/2}|P(t,t\lambda)| < B t^{-r}.
 \end{equation}
for this, we may assume without loss of generality that $M=\N^r$. It suffices to prove the bound for each 
monomial $x_0^{d_0} x_1^{d_1} \cdots x_r^{d_r}$.
We fix $c$ and $C$ such that 
$$C(x_1^2+\cdots+x_r^2) > Q(x_1,\ldots,x_r) > c(x_1^2+\cdots+x_r^2).$$
Therefore,
\begin{align*}
\sum_{\lambda\in \N^r} P(t,t \lambda)  e^{-Q(t \lambda)/2} &< t^{d_0+\cdots+d_r}\sum_{(n_1,\ldots,n_r)\in \N^r}n_1^{d_1}\cdots n_r^{d_r}e^{-ct^2(n_1^2+\cdots+n_r^2)/2} \\
&= t^{d_0+\cdots+d_r}\prod_{i=1}^r  \sum_{n=0}^\infty n^{d_i} e^{-ct^2n_i^2/2}\\
&=  t^{d_0+\cdots+d_r}\prod_{i=1}^r \Bigl( \int_0^\infty x^{d_i} e^{-ct^2x^2/2}dx+O(1)\Bigr)\\
&=  t^{d_0+\cdots+d_r}\prod_{i=1}^r \Bigl( t^{-1}\int_0^\infty (x/t)^{d_i} e^{-cx^2/2}dx+O(1)\Bigr)\\
&< B_0 t^{-r},
\end{align*}
for all sufficiently small $t$ if $B_0$ is large enough. 

For the lower bound, we first observe that $P(0,\vec x)$ is not identically zero, and since it is non-negative on $[0,\infty)^r$, we can
choose a non-empty open set $U$ 
in $\R^r$ on which $P(0,\vec x)$ can be bounded below by a constant $h>0$.
We then choose linearly independent 
$\vec v_1,\ldots,\vec v_r\in \delta+\N^r$ and intervals $(a_i,b_i)\subset (0,\infty)$ such that 
if $x_i\in (a_i,b_i)$ for $i=1,\ldots,r$, then $\sum_i x_i \vec v_i \in U$. 

We write $P$ as a sum of $P(0,\vec x)$ and a polynomial of the form $x_0 P_0$.
The sum \eqref{absolute ub} for $P_0$ has an upper bound of the form $B t^{-r}$, so the sum for $x_0P_0$ has upper bound $B t^{1-r}$.
Therefore,
\begin{align*}
\sum_{\lambda\in \N^r} (P(t,t\lambda)) e^{-Q(t \lambda)}
&\ge \sum_{\lambda\in \N^r} (P(0,t\lambda)) e^{-Q(t \lambda)} - Bt^{1-r}\\
&\ge h\!\!\! \sum_{x_1\in t^{-1}[a_1,b_1]}\!\!\!\cdots\!\!\!\sum_{x_r\in t^{-1}[a_r,b_r]} e^{-Ct^2\sum_i x_i^2/2}-Bt^{1-r}\\
& = h\prod_{i=1}^r \sum_{x\in [a_i t^{-1},b_i t^{-1}]} e^{-Ct^2 x^2/2} - Bt^{1-r}\\
& \ge h\prod_{i=1}^r (t^{-1}(b_i-a_i)-1) e^{-Cb_i^2/2} > A_0t^{-r} - Bt^{1-r}
\end{align*}
for all sufficiently small $t$ if $A_0< he^{-C\sum_i b_i^2/2} \prod_i (b_i-a_i)$.
Setting $A = A_0/2$ and assuming $t \in (0,A_0/2B)$, we get the desired lower bound.

\end{proof}

We now prove the main theorem.

\begin{proof}
Using a basis for $\R^r = X\otimes \R$ consisting of the fundamental weights of the universal cover $\tilde G^c$ of $G^c$, we can identify $\N^r$ with the set of dominant weights of $\tilde G^c$
and $M = \delta+\Lambda$ with a subsemigroup of $\N^r$ which spans $\R^r$.
We have
$$b_n^{G,V} = \sum_{\lambda\in M} c_\lambda^{G,V^{\otimes n}}.$$
Fix $N>r+u$.
Setting $t = 1/\sqrt n$ and defining $Q$ by \eqref{Edgeworth}, Proposition~\ref{c_lambda estimate} implies
$$c_\lambda^{G,V^{\otimes n}} (\dim V)^{-n} =  n^{-r/2-u/2}e^{-Q(n^{-1/2}\lambda)}P(n^{-1/2},n^{-1/2} \lambda) + o(n^{-N}).$$
As $N>r+u/2$, the sum of the error terms over $\N^r$ is $o(n^{-u/2})$.
By Proposition~\ref{AB bounds}, the sum of the main term is between $A n^{-u/2}$ and $B n^{-u/2}$, so we are done.
\end{proof}


\begin{thebibliography}{99}


\bibitem{BR}
Bhattacharya, R. N.; Ranga Rao, R.:
Normal approximation and asymptotic expansions.
Wiley Series in Probability and Mathematical Statistics. John Wiley \& Sons, New York-London-Sydney, 1976.

\bibitem{Biane}
Biane, Philippe:
Estimation asymptotique des multiplicit\'es dans les puissances tensorielles d'un g-module. 
\textit{C.\ R.\ Acad.\ Sci.\ Paris S\'er.\ I Math.} \textbf{316} (1993), no.\ 8, 849--852.

\bibitem{Borel}
Borel, Armand: Linear algebraic groups. Second edition. Graduate Texts in Mathematics, 126. Springer-Verlag, New York, 1991.

\bibitem{Bt}
Br\"ocker, Theodor; tom Dieck, Tammo:
Representations of compact Lie groups. Translated from the German manuscript. Corrected reprint of the 1985 translation. Graduate Texts in Mathematics, 98. Springer-Verlag, New York, 1995.

\bibitem{Clifford}
Clifford, A. H.:
Representations induced in an invariant subgroup.
\textit{Ann.\ of Math.\ (2)} \textbf{38} (1937), no.\ 3, 533--550.

\bibitem{CEOT}
Coulembier Kevin; Etingof, Pavel; Ostrik, Victor; Tubbenhauer, Daniel:
Fractal behavior of tensor powers of the two dimensional space in prime characteristic,
arXiv:2405.16786.

\bibitem{COT}
Coulembier, Kevin; Ostrik, Victor; Tubbenhauer, Daniel:
Growth rates of the number of indecomposable summands in tensor powers. 
\textit{Algebr.\ Represent.\ Theory} \textbf{27} (2024), no.\ 2, 1033--1062.

\bibitem{DG}
Demazure, Michel; Gabriel, Pierre: Groupes alg\'ebriques. Tome I: G\'eom\'etrie alg\`ebrique, g\'en\'eralit\'es, groupes commutatifs. 
Avec un appendice Corps de classes local par Michiel Hazewinkel. Masson \& Cie, Paris; North-Holland Publishing Co., Amsterdam, 1970.
 
\bibitem{Larsen}
Larsen, Michael:
Bounds for $\SL_2$-indecomposables in tensor powers of the natural representation in characteristic $2$,
\textit{Ann.\ Represent.\ Theory} \textbf{2} (2025), no.\ 4, 575--598.

\bibitem{GIT}
Mumford, D.; Fogarty, J.; Kirwan, F.:
Geometric invariant theory. Third edition. Ergebnisse der Mathematik und ihrer Grenzgebiete (2), 34. Springer-Verlag, Berlin, 1994.

\bibitem{PR}
Postnova, Olga; Reshetikhin, Nicolai:
On multiplicities of irreducibles in large tensor product of representations of simple Lie algebras. 
\textit{Lett.\ Math.\ Phys.} \textbf{110} (2020), no.\ 1, 147--178. 

\bibitem{Sheu}
Sheu, Nai-Heng:
Asymptotic Growth of Trivial Summands in Tensor Powers,
arXiv:2501.11125.

\bibitem{TZ}
Tate, Tatsuya; Zelditch, Steve:
Lattice path combinatorics and asymptotics of multiplicities of weights in tensor powers.
\textit{J.\ Funct.\ Anal.} \textbf{217} (2004), no.\ 2, 402--447.

\end{thebibliography}
\end{document}